\documentclass[12pt]{article}

\usepackage{amsmath,amssymb,amsthm}
\usepackage{mathrsfs}
\usepackage{authblk}

\usepackage[left=2cm,right=2cm,top=2cm,bottom=2cm,bindingoffset=0cm]{geometry}
\usepackage{nccfoots}
\usepackage[hang,flushmargin]{footmisc} 
\usepackage{hyperref}

\usepackage{graphicx}
\usepackage{pgf,tikz}

\newtheorem{proposition}{Proposition}
\newtheorem{lemma}{Lemma}
\newtheorem{corollary}{Corollary}

\newtheorem{theorem}{Theorem}
\newtheorem{problem}{Problem}
\newtheorem{definition}{Definition}

\newtheorem{question}{Question}

\newcommand{\St}{{\mathcal St}}

\title{A short proof of Oblakov's theorem}

\author[1]{Danila Cherkashin}

\affil[1]{Institute of Mathematics and Informatics, Bulgarian Academy of Sciences, Sofia}

\begin{document}

\maketitle

\begin{abstract}
We give a short proof that for a given planar set of terminals $P$ there is at most one locally minimal tree with prescribed directions at $P$ (i.e. two locally minimal trees cannot coincide in $B_\varepsilon(P)$).

The new ingredient is a combinatorial result which says that a certain embedding of a bipartite cubic graph has the same numbers of balanced and unbalanced vertices.
\end{abstract}

\section{Introduction}

Consider the following form of the Steiner (tree) problem:
\begin{problem}\label{Problem1}
For a given finite set $P = \{x_1,\dots ,x_n\} \subset \mathbb{R}^2$ to find a connected set $\St$ of minimal length (one-dimensional Hausdorff measure) containing $P$.
\end{problem}

A solution to Problem~\ref{Problem1} is called a \textit{Steiner tree}.
It is known that such an $\St = \St(P)$ always exists (but is not necessarily unique) and that it is a union of a finite set of segments. Moreover, $\St$ can be represented as a graph, embedded into the Euclidean plane, such that its set of vertices contains $P$ and all its edges are straight line segments. This graph is connected and does not contain cycles, i.e. is a tree, which explains the naming of $\St$. It is known that the maximal degree of the vertices of $\St$ is at most $3$. Moreover, only vertices $x_i$ can have degree $1$ or $2$, all the other vertices have degree $3$ and are called \textit{Steiner points} while the vertices $x_i$ are called \textit{terminals}. 
Vertices of degree $3$ are called \textit{branching points}.
The angle between any two adjacent edges of $\St$ is at least $2\pi/3$.
That means that for a branching point the angle between any two segments incident to it is exactly $2\pi/3$.

The number of Steiner points in $\St$ does not exceed $n-2$. A Steiner tree with exactly $2n-2$ vertices is called \textit{full}.  
Every terminal point of a full Steiner tree has degree one.

For a given finite set $P$ consider a connected acyclic set $S$ containing $P$.
Then $S$ is called a \textit{locally minimal tree} if $\overline{S \cap B_\varepsilon (x)}$ is a Steiner tree for 
$(\{x\} \cap P) \cup (S \cap \partial B_\varepsilon (x))$ for every point $x \in S$ and small enough $\varepsilon>0$. Clearly every Steiner tree is locally minimal and not vice versa.
Locally minimal trees have all the mentioned properties of Steiner trees except the minimal length condition. So locally minimal trees inherit the definitions of terminals, Steiner points, branching points and fullness.
A proof of the listed properties of Steiner and locally minimal trees together with additional information on them can be found in the
book~\cite{hwang1992steiner} and in the article~\cite{gilbert1968steiner}.

The Steiner problem may have several solutions starting with $n = 4$ (see Fig.~\ref{fig1}).

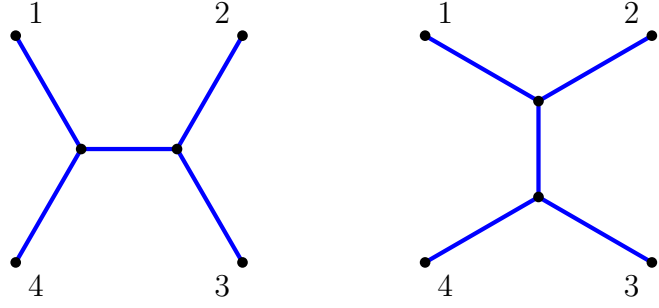
\begin{figure}
    \centering
    \begin{tikzpicture}
    \def\r{1.5cm}
    \draw[ultra thick, blue]
        (-\r, \r) coordinate(x1) node[black, above right]{$1$} --++ (-60:{\r/cos(30)}) coordinate (x5);
    \draw[ultra thick, blue]
        (\r,\r) coordinate(x2) node[black, above left]{$2$} --++ (-120:{\r/cos(30)}) coordinate (x6);
    \draw[ultra thick, blue]
        (\r, -\r) coordinate(x3) node[black, below left]{$3$} --++ (120:{\r/cos(30)});
    \draw[ultra thick, blue]
        (-\r,-\r) coordinate(x4) node[black, below right]{$4$} --++ (60:{\r/cos(30)});
    \draw[ultra thick, blue]
        (x5) -- (x6);
    \foreach \x in{1,2,...,6}{
        \fill (x\x) circle (2pt);
    }
\end{tikzpicture}
\hspace{2cm}
\begin{tikzpicture}
    \def\r{1.5cm}
    \draw[ultra thick, blue]
        (-\r, \r) coordinate(x1) node[black, above right]{$1$} --++ (-30:{\r/cos(30)}) coordinate (x5);
    \draw[ultra thick, blue]
        (\r,\r) coordinate(x2) node[black, above left]{$2$} --++ (-150:{\r/cos(30)});
    \draw[ultra thick, blue]
        (\r, -\r) coordinate(x3) node[black, below left]{$3$} --++ (150:{\r/cos(30)}) coordinate (x6);
    \draw[ultra thick, blue]
        (-\r,-\r) coordinate(x4) node[black, below right]{$4$} --++ (30:{\r/cos(30)});
    \draw[ultra thick, blue]
        (x5) -- (x6);
    \foreach \x in{1,2,...,6}{
        \fill (x\x) circle (2pt);
    }
\end{tikzpicture}
    \caption{An example of a non-unique solution. The labeled points form a square.}
    \label{fig1}
\end{figure}

A \textit{topology} $T$ of a labeled Steiner tree (or a labeled locally minimal tree) $\St$ is the corresponding abstract graph with labeled terminals and unlabeled Steiner points.

\begin{proposition}[Melzak,~\cite{melzak1961problem}]
For every topology $T$ there is at most one locally minimal tree for a configuration $P$ with topology $T$. 
\label{melzakuniq}
\end{proposition}

Proposition~\ref{melzakuniq} allows us to define $S_T(P)$ as the locally minimal tree, if it exists.
We say that trees $S_{T_1}(P)$ and $S_{T_2}(P)$ are \textit{codirected at terminals} if for some $\varepsilon > 0$ we have
\[
S_{T_1}(P) \cap B_\varepsilon (P) = S_{T_2}(P) \cap B_\varepsilon (P), 
\]
where $B_\varepsilon(P)$ is an open $\varepsilon$-neighborhood of the terminal set $P$.
The following theorem generalizes the results from~\cite{ivanov2006uniqueness} from global minimizers to local minimizers.

\begin{theorem}[Oblakov~\cite{oblakov2009non}] \label{th:oblakov}
There are no two distinct topologies $T_1$ and $T_2$ and a planar configuration $P$ such that locally minimal trees $S_{T_1}(P)$ and $S_{T_2}(P)$ exist and are codirected at terminals.
\end{theorem}

One of the applications of Theorem~\ref{th:oblakov} is that for $n \geq 4$ the set of $n$-point configurations for which the solution to the Steiner problem is not unique has the Hausdorff dimension at most $2n-1$ (as a subset of $\mathbb{R}^{2n}$), see~\cite{basok2024uniqueness}.

\section{New combinatorial argument}

The following combinatorial argument is self-contained and is not directly related to Steiner trees.
For a drawn cubic graph let us say that a vertex $v$ is a \textit{claw} 
if the three edges incident to $v$ lie in a common open half-plane whose boundary passes through $v$
and a \textit{tripod} otherwise, see Fig.~\ref{fig:clawpod}.

\begin{figure}[h]
    \centering
    \begin{tikzpicture}

  \begin{scope}
  
    \coordinate (a) at (0,0);
    \coordinate (b) at (0,-1);
    \coordinate (c) at (-0.866,-0.5);
    \coordinate (d) at (0.866,-0.5);

    \draw [ultra thick, blue] (c) -- (a) -- (d);
    \draw [ultra thick, blue] (b) -- (a);

    \filldraw (a) circle (2pt);

\end{scope}

  \begin{scope}[shift={(5,-0.5)}]

\coordinate (a) at (0,0);
    \coordinate (b) at (0,1);
    \coordinate (c) at (-0.866,-0.5);
    \coordinate (d) at (0.866,-0.5);

    \draw [ultra thick, blue] (c) -- (a) -- (d);
    \draw [ultra thick, blue] (b) -- (a);

    \filldraw (a) circle (2pt);
    
  \end{scope}
    
\end{tikzpicture}
    \caption{Local structure of $H$. A claw and a tripod}
    \label{fig:clawpod}
\end{figure}
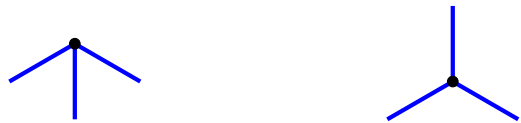

\begin{theorem} \label{th:comb}
    Let $H$ be a cubic bipartite plane graph and let $l_1$, $l_2$, $l_3$ be pairwise non-parallel lines.
    Suppose that $H$ is drawn in a way that every vertex has incident edges parallel to $l_1$, $l_2$ and $l_3$ (one edge in each direction). 
    Then the number of claws in the drawing is equal to the number of tripods.
\end{theorem}

From the point of view of differential geometry (see for instance~\cite{frolov2026geodesic}) after a proper affine transform a claw becomes an unbalanced vertex and
a tripod becomes a balanced vertex. So Theorem~\ref{th:comb} says that for a very specific type of network the number of balanced vertices is equal to the number of unbalanced vertices. In Section~\ref{sect:ex} we give examples that show that this class of networks is quite diverse.

\begin{proof}
Denote by $X$ and $Y$ the bipartition classes of $H$.
After applying an affine transformation, we may assume that
the three edge directions make angles that are multiples of $\pi/3$.

Let $M$ be the set of all edges parallel to $l_1$, and put $H' = H\setminus M$.
Every vertex of $H'$ has degree two: it is incident to one edge parallel
to $l_2$ and one edge parallel to $l_3$. Hence $H'$ is a disjoint union
of cycles.

Fix a cycle $C$ of $H'$ and let $D_C$ be
the bounded component of $\mathbb R^2\setminus C$. Write $|C|=2k$.
Label the vertices of $C$ by $a,b,c,d$ according as the interior angle
of $D_C$ is
\[
    \frac{\pi}{3},\quad
    \frac{2\pi}{3},\quad
    \frac{4\pi}{3},\quad
    \frac{5\pi}{3},
\]
respectively. Since the edges parallel to $l_2$ and $l_3$ alternate along $C$, the
vertices labeled $a$ or $c$ form one bipartition class of $C$, while
the vertices labeled $b$ or $d$ form the other one. Consequently,
\begin{equation} \label{eq:comb1}
    \#a+\#c = |V(C) \cap X| = |V(C) \cap Y| = \#b+\#d = k.    
\end{equation}

For $x\in\{a,b,c,d\}$, mark a vertex by $x^+$ if its incident edge from
$M$ lies in $D_C$, and by $x^-$ otherwise. 

\begin{lemma}
    For every cycle $C$ in $H'$ we have
    \begin{equation} \label{eq:comb2}
    \#a^+ + \#c^+ = \#b^+ + \#d^+.
    \end{equation}
\end{lemma}

\begin{proof}
Let $U$ be the set of all vertices lying
strictly inside $C$. Since $H'$ is a disjoint union of cycles, $U$ is
the union of the vertex sets of some cycles of $H'$. Thus
\begin{equation} \label{eq:comb3}
    |U\cap X|=|U\cap Y|.
\end{equation}

Let $e_U$ be the number of edges with both endpoints in $U$, and let
$q_X$ and $q_Y$ be the numbers of edges joining $U\cap X$ and
$U\cap Y$, respectively, to $C$. By planarity, every edge leaving $U$
ends on $C$ and since $H$ is bipartite every edge counted by $e_U$ connects $U\cap X$ with $U\cap Y$. Counting the degrees of the vertices of $U\cap X$ and $U\cap Y$, respectively, in the graph $H$ 
gives
\[
    3|U\cap X|=e_U+q_X,
    \qquad
    3|U\cap Y|=e_U+q_Y.
\]
Together with~\eqref{eq:comb3}, this implies $q_X=q_Y$.

Therefore, the edges of $M$ joining $C$ to vertices strictly inside
$C$ have equally many endpoints in the two bipartition classes of
$C$. An edge of $M$ whose two endpoints lie on $C$ and whose interior
is contained in $D_C$ also has one endpoint in each bipartition class.
The lemma is proven.
\end{proof}

A direct inspection of the four local configurations shows that
$a^+=0$, $d^-=0$, and that a vertex of $C$ is a tripod precisely when it is of type
$b^-$ or $c^+$. It follows from~\eqref{eq:comb2} that
\[
    \#c^+=\#b^++\#d^+=\#b^++\#d.
\]
Hence
\[
    N_{\rm tripod}(C) = \#b^-+\#c^+ =\#b^-+\#b^++\#d = \#b+\#d = k
\]
by~\eqref{eq:comb1}. Since $C$ has $2k$ vertices, it also contains exactly $k$
claws.

Every vertex of $H$ belongs to exactly one cycle of $H'$. Summing over
all cycles, we obtain
\[
    N_{\rm claw}(H)=N_{\rm tripod}(H).
\]
\end{proof}

\section{Proof of Theorem~\ref{th:oblakov}}

Oblakov~\cite{oblakov2009non} provided the proof which starts with a smart restructuring and continues by a rather foggy invariant argument. 
We replace the second part with an elementary combinatorial Theorem~\ref{th:comb} and some elementary observations.
For the sake of completeness we repeat the beginning of the proof.

\subsection{Oblakov's restructuring}

Suppose that there are two locally minimal trees (blue and red) which are codirected at the terminals. Then consider their symmetric difference $G_0$ and choose a maximal (with respect to inclusion) subset of segments $G$ in $G_0$ 
that are parallel to some lines $l_1$, $l_2$, $l_3$ with pairwise angles $\pi/3$. 
The local structure of $G$ is depicted in Fig.~\ref{fig:localstr} (everything has red-blue and rotation symmetries).
One can identify $G$ with a graph, whose vertices are the ends of maximal with respect to inclusion segments and their intersection point and segments being the edges. The edges inherit the red-blue coloring from segments.

Define $G_1$ and $G_2$ as blue and red subgraphs of $G$; by the setup $G_1$ and $G_2$ are forests. 
If $G$ has a part of the first or the second type from Fig.~\ref{fig:localstr}, then we replace every such local pattern as shown in Fig.~\ref{fig:replacement}. After the replacement we have a new graph $H$ with red and blue subgraphs $H_1$ and $H_2$.
Clearly, $H$ inherits the following key properties of $G$: it is a finite plane cubic straight-line graph; at every vertex there is one edge in each of three directions and both $H_1$, $H_2$ are forests.
Thus $H$ satisfies all the geometric assumptions of Theorem~\ref{th:comb} (bipartiteness is given later by Lemma~\ref{lm:bipartite}), in particular a neighborhood of every vertex is a claw or a tripod (see Fig.~\ref{fig:clawpod}). Note that if $H \neq G$ then $H$ has more vertices than $G$ and the number of trees in forests $H_1$ and $H_2$ is greater than in $G_1$ and $G_2$.

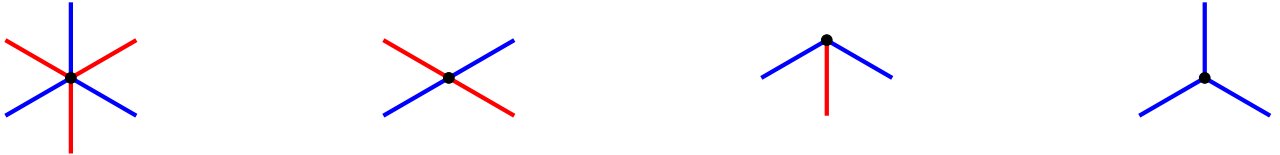
\begin{figure}[h]
    \centering
    \begin{tikzpicture}

  \begin{scope}
  
    \coordinate (a) at (0,0);
    \coordinate (b) at (0,-1);
    \coordinate (c) at (-0.866,-0.5);
    \coordinate (d) at (0.866,-0.5);
    \coordinate (e) at (0,1);
    \coordinate (f) at (0.866,0.5);
    \coordinate (g) at (-0.866,0.5);

    \draw [ultra thick, blue] (c) -- (a) -- (d);
    \draw [ultra thick, red] (b) -- (a) -- (f);
    \draw [ultra thick, red] (a) -- (g);
    \draw [ultra thick, blue] (a) -- (e);

    \filldraw (a) circle (2pt);

\end{scope}

  \begin{scope}[shift={(5,0)}]

    \coordinate (a) at (0,0);
    \coordinate (c) at (-0.866,-0.5);
    \coordinate (d) at (0.866,-0.5);
    \coordinate (f) at (0.866,0.5);
    \coordinate (g) at (-0.866,0.5);

    \draw [ultra thick, blue] (c) -- (a) -- (f);
    \draw [ultra thick, red] (d) -- (a) -- (g);

    \filldraw (a) circle (2pt);
    
  \end{scope}

  \begin{scope}[shift={(10,0.5)}]
  
    \coordinate (a) at (0,0);
    \coordinate (b) at (0,-1);
    \coordinate (c) at (-0.866,-0.5);
    \coordinate (d) at (0.866,-0.5);

    \draw [ultra thick, blue] (c) -- (a) -- (d);
    \draw [ultra thick, red] (b) -- (a);

    \filldraw (a) circle (2pt);

\end{scope}

  \begin{scope}[shift={(15,0)}]

\coordinate (a) at (0,0);
    \coordinate (b) at (0,1);
    \coordinate (c) at (-0.866,-0.5);
    \coordinate (d) at (0.866,-0.5);

    \draw [ultra thick, blue] (c) -- (a) -- (d);
    \draw [ultra thick, blue] (b) -- (a);

    \filldraw (a) circle (2pt);
    
  \end{scope}
    
\end{tikzpicture}
    \caption{Local structure of $G$}
    \label{fig:localstr}
\end{figure}

\begin{figure}
    \centering
    \begin{tikzpicture}[scale=2]

  \begin{scope}
    \coordinate (a) at (0,0);
    \coordinate (b) at (0,-1);
    \coordinate (c) at (-0.866,-0.5);
    \coordinate (d) at (0.866,-0.5);
    \coordinate (e) at (0,1);
    \coordinate (f) at (0.866,0.5);
    \coordinate (g) at (-0.866,0.5);

    \draw [ultra thick, blue] (c) -- (a) -- (d);
    \draw [ultra thick, red] (b) -- (a) -- (f);
    \draw [ultra thick, red] (a) -- (g);
    \draw [ultra thick, blue] (a) -- (e);

    \filldraw (a) circle (1.1pt);

   \end{scope}

\draw [->,>=stealth] (0.7,.0) -- (1.2,.0);

   \begin{scope}[shift={(2,0)}]
    \coordinate (a) at (0,0);
    \coordinate (b) at (0,-1);
    \coordinate (c) at (-0.866,-0.5);
    \coordinate (d) at (0.866,-0.5);
    \coordinate (e) at (0,1);
    \coordinate (f) at (0.866,0.5);
    \coordinate (g) at (-0.866,0.5);

    \draw [ultra thick, blue] (e) -- (a) --++ (-150:0.2) coordinate (a1);
    \draw [ultra thick, blue] (a) --++ (-30:0.2) coordinate (a2);
    \draw [ultra thick, red] (a1) --++ (90:0.2) coordinate (g1) -- (g);
    
    \draw [ultra thick, blue] (g1) --++ (-150:0.6) coordinate (c1) --++ (-90:0.2) coordinate (c2) -- (c);
    \draw [ultra thick, blue] (c2) --++ (-30:0.6) coordinate (b1) --++ (30:0.2) coordinate (b2) --++ (-30:0.2) coordinate (b3);
    \draw [ultra thick, red] (b2) -- (b);
    \draw [ultra thick, blue] (b3) --++ (30:0.6) coordinate (d1) -- (d); 
    \draw [ultra thick, blue] (d1) --++ (90:0.2) coordinate (d2) --++ (150:0.6) coordinate (f1); 
    \draw [ultra thick, red] (f) -- (f1) -- (a2);
    \draw [ultra thick, red] (a1) --++ (-30:0.2) coordinate (o1) --++ (-90:0.2) coordinate (o2) --++ (-150:0.2) coordinate (o3) -- (c1);
    \draw [ultra thick, red] (o3) -- (b1);
    \draw [ultra thick, red] (o2) --++ (-30:0.2) coordinate (o4) -- (b3);
    \draw [ultra thick, red] (o1) -- (a2);
    \draw [ultra thick, red] (o4) -- (d2);
 
    \filldraw (a) circle (1pt);
    \filldraw (a1) circle (1pt);
    \filldraw (a2) circle (1pt);
    \filldraw (b1) circle (1pt);
    \filldraw (b2) circle (1pt);
    \filldraw (b3) circle (1pt);
    \filldraw (c1) circle (1pt);
    \filldraw (c2) circle (1pt);
    \filldraw (d1) circle (1pt);
    \filldraw (d2) circle (1pt);
    \filldraw (f1) circle (1pt);
    \filldraw (g1) circle (1pt);
    \filldraw (o1) circle (1pt);
    \filldraw (o2) circle (1pt);
    \filldraw (o3) circle (1pt);
    \filldraw (o4) circle (1pt);

  \end{scope}

    \begin{scope}[shift={(4.5,0)}]
    \coordinate (b) at (0,-1);
    \coordinate (c) at (-0.866,-0.5);
    \coordinate (e) at (0,1);
    \coordinate (f) at (0.866,0.5);

    \draw [ultra thick, blue] (b) -- (e);
    \draw [ultra thick, red] (c) -- (f);
 

   \end{scope}

\draw [->,>=stealth] (5.1,.0) -- (5.6,.0);

   \begin{scope}[shift={(6.3,0)}]
    
    \coordinate (a) at (0,0);
    \coordinate (b) at (0,-1);
    \coordinate (c) at (-0.866,-0.5);
    \coordinate (d) at (0.866,-0.5);
    \coordinate (e) at (0,1);
    \coordinate (f) at (0.866,0.5);
    \coordinate (g) at (-0.866,0.5);

    \draw [ultra thick, blue] (e) -- (a);
    \draw [ultra thick, red] (a) -- (f);  
    
    \draw [ultra thick, red] (a) --++ (150:0.2) coordinate (a1) --++ (-150:0.4) coordinate (c1) --++ (-90:0.2) coordinate (c2) -- (c);
    \draw [ultra thick, red] (c2) --++ (-30:0.2) coordinate (c3) --++ (30:0.4) coordinate (b1);
    \draw [ultra thick, blue] (b1) -- (b);
    \draw [ultra thick, blue] (a1) --++ (-90:0.2) coordinate (o1) --++ (-150:0.2) coordinate (o2) -- (c1);
    \draw [ultra thick, blue] (o1) -- (b1);
    \draw [ultra thick, blue] (o2) -- (c3);
     
    \filldraw (a) circle (1pt);
    \filldraw (a1) circle (1pt);
    \filldraw (b1) circle (1pt);
    \filldraw (c1) circle (1pt);
    \filldraw (c2) circle (1pt);
    \filldraw (c3) circle (1pt);
    \filldraw (o1) circle (1pt);
    \filldraw (o2) circle (1pt);

  \end{scope}
    
\end{tikzpicture}
    \caption{Replacements in $G$}
    \label{fig:replacement}
\end{figure}
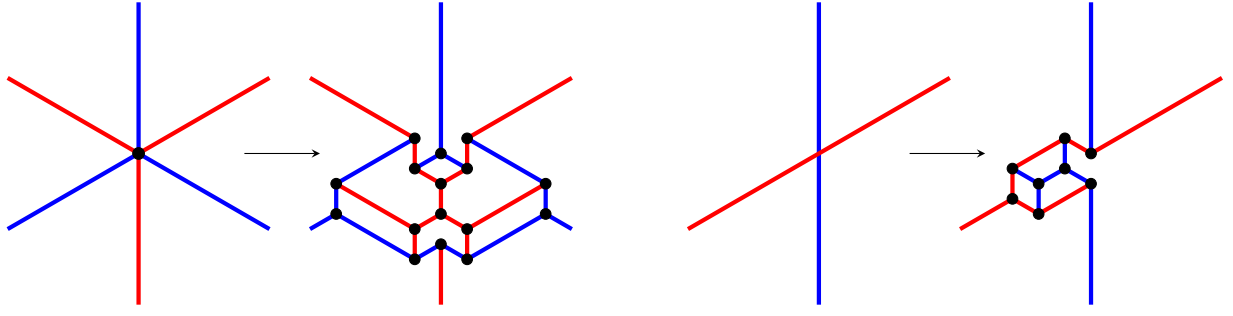

\subsection{Completion of the proof}

First recall that after the Oblakov restructuring we obtained a plane cubic graph $H$.

\begin{lemma} \label{lm:bipartite}
 The graph $H$ is bipartite.
\end{lemma}

\begin{proof}
Consider any simple cycle $C$. Let $x,y,z,w$ be the number of angles of size $2\pi/3$, $\pi/3$, $4\pi/3$ and $5\pi/3$ in $C$, respectively.
Then the sum of signed exterior turning angles on $C$ is 
\[
\frac{\pi}{3}x + \frac{2\pi}{3}y - \frac{\pi}{3}z - \frac{2\pi}{3}w =  \pm 2\pi.
\]
Thus $x + z$ is even. Note that $y + w$ is exactly the number of color exchanges along the boundary of $C$ which is also an even number. Thus $x + y + z + w$ is even as desired.

Since every cycle has even length, the graph $H$ is bipartite.
\end{proof}

Recall that $H_1$ and $H_2$ are the red and blue subgraphs of $H$.

\begin{lemma} \label{lm:notequal}
Let $t$ be the total number of trees with at least two vertices in $H_1$ and $H_2$. Then the number of claws is equal to the number of tripods plus $2t$.
\end{lemma}

\begin{proof}
For every non-empty tree $T$ of $H_1$ or $H_2$ one has $d_1(T) = d_3(T) + 2$, where $d_i(T)$ stands for the number of vertices with degree $i$ in $T$.
Then sum these values over all trees in $H_1$ and $H_2$.
Note that every claw is counted exactly once in $d_1$ (because the middle edge in a claw has not the same color as the outer ones, see Fig.~\ref{fig:localstr}) and every tripod is counted in exactly one tree as a part of $d_3$.
\end{proof}

Now we are ready to collect the parts of the main proof together.

\begin{proof}[Proof of Theorem~\ref{th:oblakov}]
   First apply Oblakov reduction to obtain a plane cubic graph $H$. By Lemma~\ref{lm:bipartite} $H$ is bipartite. Then we forget the red-blue coloring of the edges of $H$ and apply Theorem~\ref{th:comb} which gives that the number of claws is equal to the number of tripods in $H$. Finally, Lemma~\ref{lm:notequal} gives that the number of trees in red and blue parts of $H$ is zero and thus $H$ is empty. 
\end{proof}

\section{Examples} \label{sect:ex}

Here we collect examples which show that the class of cubic plane $H$ satisfying the conditions of Theorem~\ref{th:comb} is quite wide.
Then we give examples of non-bipartite $H$ with 
\[
 N_{\rm claw}(H) - N_{\rm tripod}(H) = 4k
\]
for any $k \in \mathbb{Z}$. In Section~\ref{sec:degree} we show that four always divides this difference.

\subsection{Bipartite examples}

The smallest bipartite example has 12 vertices. Indeed, there is a classification of small connected cubic bipartite graphs \href{https://oeis.org/A006823}{https://oeis.org/A006823} and it is well known that the only cubic bipartite planar graph with at most 10 vertices is the cube graph $Q_3$. By Theorem~\ref{th:comb} any drawing of $Q_3$ has 4 claws and 4 tripods. Thus these 4 claws form the outer quadrilateral face; so the sum of the outer angles of the outer face is $4 \cdot \pi/3 \neq 2\pi$, a contradiction.
An example with 12 vertices is drawn in Figure~\ref{fig:12ex}.

\begin{figure}[h]
    \centering
    \begin{tikzpicture}[
    x={(1cm,0cm)},
    y={({0.5cm},{0.8660254cm})},
    line cap=round,
    line join=round,
    scale=1
]


\coordinate (A) at (-2,2);
\coordinate (B) at (0,2);
\coordinate (C) at (2,0);
\coordinate (D) at (2,-2);
\coordinate (E) at (0,-2);
\coordinate (F) at (-2,0);

\coordinate (A1) at (-4,4);
\coordinate (B1) at (0,4);
\coordinate (C1) at (4,0);
\coordinate (D1) at (4,-4);
\coordinate (E1) at (0,-4);
\coordinate (F1) at (-4,0);


\draw[ultra thick, blue] (A) -- (B) -- (C) -- (D) -- (E) -- (F) -- (A);
\draw[ultra thick, blue] (A1) -- (B1) -- (C1) -- (D1) -- (E1) -- (F1) -- (A1);

\foreach \P in {A,B,C,D,E,F} 
    \draw[ultra thick, blue] (\P) -- (\P1);

\foreach \P in {A,B,C,D,E,F}
    \filldraw[black] (\P) circle (2pt);
\foreach \P in {A,B,C,D,E,F}    
    \filldraw[black] (\P1) circle (2pt);

\end{tikzpicture}
    \caption{An example of plane cubic graphs $H$ with 6 claws and 6 tripods}
    \label{fig:12ex}
\end{figure}
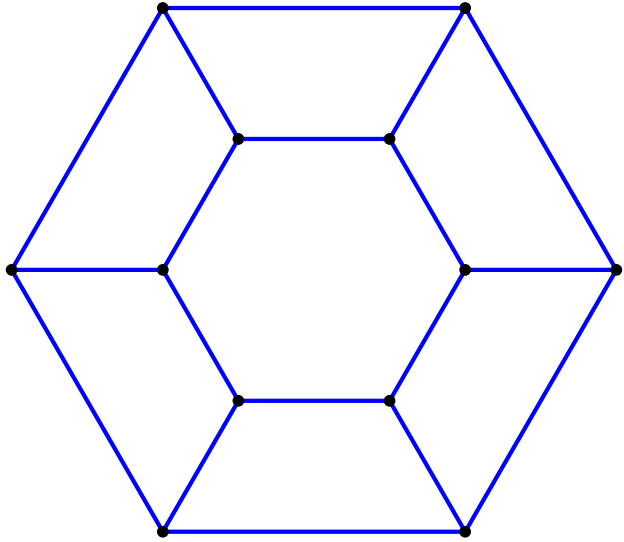

\subsection{Non-bipartite examples}

Again the classifications gives that the only cubic planar graphs with at most 6 vertices are $K_4$ and the triangular prism.
It is straightforward to check that both these graphs cannot be drawn by straight edges in three parallel directions.
Thus the smallest number of vertices in an example is 8, see Figure~\ref{fig:exampleClaws}.

\begin{figure}[h]
    \centering
    \hfill
\begin{tikzpicture}[
    x={(1cm,0cm)},
    y={({-0.5cm},{0.8660254cm})},
    line cap=round,
    line join=round,
    scale=1
]


\coordinate (A) at (-2,0);
\coordinate (B) at (0,0);
\coordinate (C) at (0,2);
\coordinate (D) at (2,2);

\coordinate (E) at (-2,-3);
\coordinate (F) at (0,-3);
\coordinate (G) at (0,-1);
\coordinate (H) at (2,-1);



\draw[ultra thick, blue] (C) -- (A) -- (B) -- (C) -- (D) -- (B);
\draw[ultra thick, blue] (A) -- (E);
\draw[ultra thick, blue] (D) -- (H);
\draw[ultra thick, blue] (G) -- (E) -- (F) -- (G) -- (H) -- (F);

\foreach \P in {A,B,C,D,E,F,G,H}
    \filldraw[black] (\P) circle (2pt);

\end{tikzpicture}~\hfill~\begin{tikzpicture}[
    x={(1cm,0cm)},
    y={({-0.5cm},{0.8660254cm})},
    line cap=round,
    line join=round,
    scale=1
]


\coordinate (A) at (-2,0);
\coordinate (B) at (0,0);
\coordinate (C) at (0,2);
\coordinate (D) at (2,2);
\coordinate (E) at (2,-1);
\coordinate (F) at (1,-1);
\coordinate (G) at (1,-2);
\coordinate (H) at (0,-2);
\coordinate (J) at (0,-1);
\coordinate (K) at (-1,-1);
\coordinate (L) at (-1,-2);
\coordinate (M) at (-2,-2);



\draw[ultra thick, blue] (A) -- (B) -- (C) -- (D) -- (E) -- (F) -- (G) -- (H) -- (J) -- (K) -- (L) -- (M) -- (A);
\draw[ultra thick, blue] (A) -- (C);
\draw[ultra thick, blue] (B) -- (D);
\draw[ultra thick, blue] (E) -- (G);
\draw[ultra thick, blue] (F) -- (H);
\draw[ultra thick, blue] (J) -- (L);
\draw[ultra thick, blue] (K) -- (M);

\foreach \P in {A,B,C,D,E,F,G,H,J,K,L,M}
    \filldraw[black] (\P) circle (2pt);

\end{tikzpicture}
\hfill
    \caption{Examples of non-bipartite plane cubic graphs $H$ with 8 and 12 claws and no tripods}
    \label{fig:exampleClaws}
\end{figure}
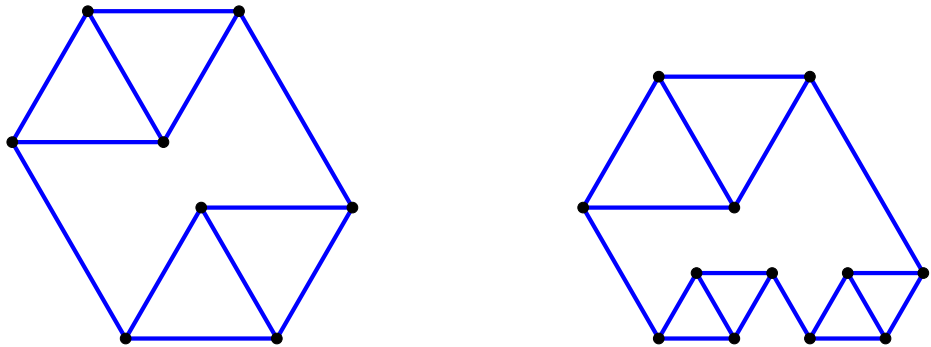

Also, the examples in Figure~\ref{fig:exampleClaws} consist only of claws.
Clearly, a vertex which is a tripod belongs to the convex hull of other vertices. Thus there is no picture consisting only of tripods.
To create a picture with many tripods consider the following pattern with 10 tripods and 2 claws, see Figure~\ref{fig:TripodsExample}.
One may glue $m$ copies of this pattern by dashed edges and then close it by up and down part from the rightmost part of Figure~\ref{fig:TripodsExample}.
Then the resulting graph has $10m + 6$ tripods and $2m + 6$ claws.

\begin{figure}[h]
    \centering
    \begin{tikzpicture}[
    x={(1cm,0cm)},
    y={({0.5cm},{0.8660254cm})},
    line cap=round,
    line join=round,
    scale=0.5
]


\coordinate (A) at (0,0);
\coordinate (B) at (2,0);
\coordinate (C) at (4,-2);
\coordinate (D) at (4,-4);
\coordinate (E) at (2,-4);
\coordinate (F) at (0,-2);

\coordinate (G) at (2,2);
\coordinate (GG) at (0,4);
\coordinate (H) at (4,2);
\coordinate (HH) at (4,4);
\coordinate (J) at (8,-2);
\coordinate (JJ) at (8,-4);

\coordinate (M) at (8,-6);
\coordinate (N) at (6,-6);

\coordinate (G1) at (2,-6);
\coordinate (H1) at (0,-6);
\coordinate (J1) at (-4,-2);
\coordinate (JJ1) at (-2,-4);

\coordinate (M1) at (-4,2);
\coordinate (N1) at (-2,2);

\coordinate (MM1) at (-6,4);
\coordinate (NN1) at (-2,4);


\draw[ultra thick, blue] (A) -- (B) -- (C) -- (D) -- (E) -- (F) -- (A);
\draw[ultra thick, blue] (B) -- (G) -- (H) -- (J) -- (C);
\draw[ultra thick, blue] (J) -- (M);
\draw[ultra thick, blue] (N) -- (D);

\draw[ultra thick, blue] (E) -- (G1);
\draw[ultra thick, blue] (H1) -- (J1) -- (F);
\draw[ultra thick, blue] (J1) -- (M1) -- (N1) -- (A);

\draw[ultra thick, blue] (G) -- (GG);
\draw[ultra thick, blue] (H) -- (HH);
\draw[ultra thick, blue] (M) -- (JJ);
\draw[ultra thick, blue] (N) -- (D);
\draw[ultra thick, blue] (G1) -- (E);
\draw[ultra thick, blue] (H1) -- (JJ1);
\draw[ultra thick, blue] (M1) -- (MM1);
\draw[ultra thick, blue] (N1) -- (NN1);
\draw[ultra thick, dashed, red] (G) -- (GG);
\draw[ultra thick, dashed, red] (H) -- (HH);
\draw[ultra thick, dashed, red] (M) -- (JJ);
\draw[ultra thick, dashed, red] (N) -- (D);
\draw[ultra thick, dashed, red] (G1) -- (E);
\draw[ultra thick, dashed, red] (H1) -- (JJ1);
\draw[ultra thick, dashed, red] (M1) -- (MM1);
\draw[ultra thick, dashed, red] (N1) -- (NN1);

\foreach \P in {A,B,C,D,E,F,G,H,J,J1,M1,N1}
    \filldraw[black] (\P) circle (4pt);

\end{tikzpicture}
\hfill
\begin{tikzpicture}[
    x={(1cm,0cm)},
    y={({0.5cm},{0.8660254cm})},
    line cap=round,
    line join=round,
    scale=0.5
]


\coordinate (A) at (0,0);
\coordinate (B) at (2,0);
\coordinate (C) at (4,-2);
\coordinate (D) at (4,-4);
\coordinate (E) at (2,-4);
\coordinate (F) at (0,-2);

\coordinate (G) at (2,2);
\coordinate (H) at (4,2);
\coordinate (J) at (8,-2);
\coordinate (JJ) at (8,-4);

\coordinate (M) at (8,-6);
\coordinate (N) at (6,-6);

\coordinate (G1) at (2,-6);
\coordinate (H1) at (0,-6);

\coordinate (G2) at (3,-8);
\coordinate (H2) at (1,-8);

\coordinate (J1) at (-4,-2);
\coordinate (JJ1) at (-2,-4);

\coordinate (M1) at (-4,2);
\coordinate (N1) at (-2,2);

\coordinate (X) at (2,4);
\coordinate (Y) at (-4,4);

\coordinate (K) at (8,-8);
\coordinate (L) at (2,-8);

\coordinate (K2) at (9,-10);
\coordinate (L2) at (3,-10);

\coordinate (M2) at (9,-8);
\coordinate (N2) at (7,-8);


\draw[ultra thick, blue] (A) -- (B) -- (C) -- (D) -- (E) -- (F) -- (A);
\draw[ultra thick, blue] (B) -- (G);
\draw[ultra thick, blue] (H) -- (J) -- (C);
\draw[ultra thick, blue] (J) -- (M);
\draw[ultra thick, blue] (N) -- (D);

\draw[ultra thick, blue] (E) -- (G1);
\draw[ultra thick, blue] (H1) -- (J1) -- (F);
\draw[ultra thick, blue] (J1) -- (M1);
\draw[ultra thick, blue] (N1) -- (A);

\draw[ultra thick, blue] (M) -- (JJ);
\draw[ultra thick, blue] (N) -- (D);
\draw[ultra thick, blue] (H1) -- (JJ1);

\draw[ultra thick, blue] (G2) -- (L2);
\draw[ultra thick, blue] (H2) -- (L2);
\draw[ultra thick, blue] (M2) -- (K2);
\draw[ultra thick, blue] (N2) -- (K2);
\draw[ultra thick, blue] (K2) -- (L2);

\draw[ultra thick, dashed, red] (G2) -- (L2);
\draw[ultra thick, dashed, red] (H2) -- (L2);
\draw[ultra thick, dashed, red] (M2) -- (K2);
\draw[ultra thick, dashed, red] (N2) -- (K2);

\draw[ultra thick, blue] (X) -- (Y);
\draw[ultra thick, blue] (G) -- (X);
\draw[ultra thick, blue] (H) -- (X);
\draw[ultra thick, blue] (M1) -- (Y);
\draw[ultra thick, blue] (N1) -- (Y);

\draw[ultra thick, dashed, red] (M) -- (JJ);
\draw[ultra thick, dashed, red] (N) -- (D);
\draw[ultra thick, dashed, red] (G1) -- (E);
\draw[ultra thick, dashed, red] (H1) -- (JJ1);

\foreach \P in {A,B,C,D,E,F,X,Y,K2,L2,J,J1}
    \filldraw[black] (\P) circle (4pt);

\end{tikzpicture}
    \caption{A construction of a graph $H$ with many tripods}
    \label{fig:TripodsExample}
\end{figure}
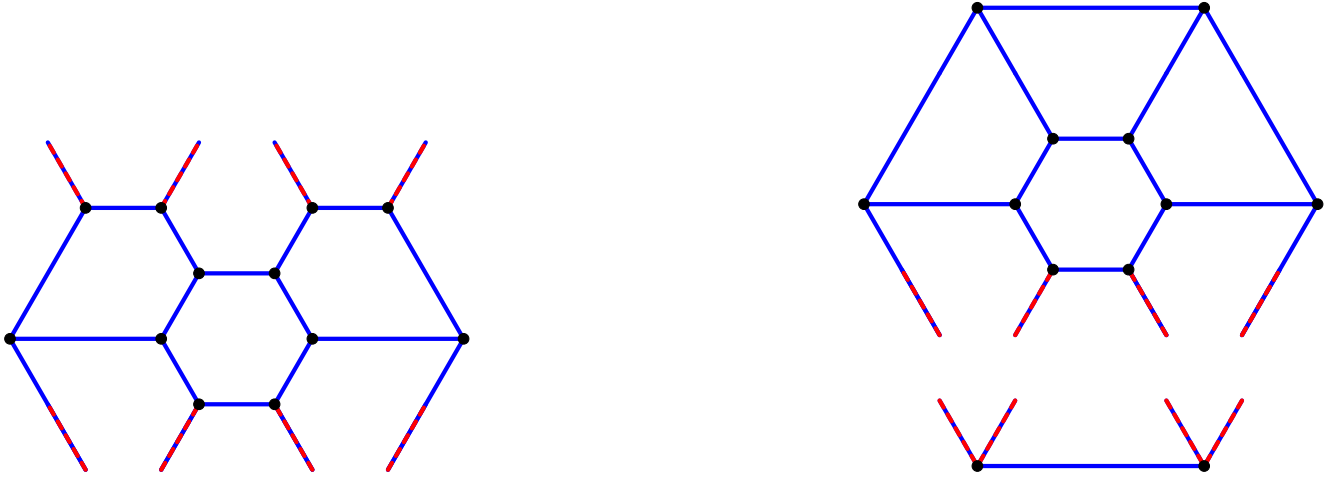

A (non-connected) union of the listed examples can reach any difference between the numbers of claws and tripods of the form $4k$, $k \in \mathbb{Z}$.

Finally, let us note that the numbers of claws and tripods can be odd, see Figure~\ref{fig:oddExample}.

\begin{figure}[h]
    \centering
    \begin{tikzpicture}[
    x={(1cm,0cm)},
    y={({-0.5cm},{0.8660254cm})},
    line cap=round,
    line join=round,
    scale=2
]

\coordinate (v0) at (0,0);
\coordinate (v1) at (1,0);
\coordinate (v2) at (0,1);
\coordinate (v3) at (1,1);
\coordinate (v4) at (3,2);
\coordinate (v5) at (1,2);
\coordinate (v6) at (3,4);
\coordinate (v7) at (2,4);
\coordinate (v8) at (1,3);
\coordinate (v9) at (2,3);

\draw[ultra thick, blue]
(v0)--(v1)
(v2)--(v3)
(v4)--(v5)
(v6)--(v7)
(v8)--(v9);

\draw[ultra thick, blue]
(v0)--(v2)
(v1)--(v3)
(v4)--(v6)
(v5)--(v8)
(v7)--(v9);

\draw[ultra thick, blue]
(v0)--(v3)
(v1)--(v4)
(v2)--(v5)
(v6)--(v9)
(v7)--(v8);

\foreach \i in {0,...,9}
\filldraw[black] (v\i) circle (1pt);

\end{tikzpicture}
    \caption{An example of non-bipartite plane cubic graphs $H$ with 9 claws and 1 tripod}
    \label{fig:oddExample}
\end{figure}
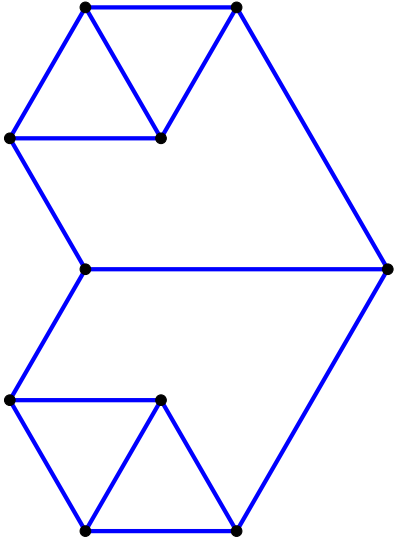

\section{Possibly non-bipartite \texorpdfstring{$H$}{H}}
\label{sec:degree}

\subsection{Degree of the face coloring}

We first recall the topological degree and its standard simplicial and cellular computation. 
The use of this construction for proper $4$-colorings
of triangulations goes back to Fisk; see
\cite{fisk1973structure,fisk1977geometric}. We follow the explicit
exposition in~\cite[Section~3.1]{mohar2009new}.

\begin{definition}[Topological degree]
\label{def:topological-degree}
Let $K$ be a triangulation of an oriented $2$-sphere, and let
\[
    \Phi\colon |K|\longrightarrow\partial\Delta^3
\]
be a continuous map. Orient $\partial\Delta^3$. The degree of
$\Phi$ is the integer $\deg\Phi$ determined by
\[
    \Phi_*[K]
    =
    \deg(\Phi)\,[\partial\Delta^3]
\]
in $H_2(\partial\Delta^3;\mathbb Z)\simeq\mathbb Z$,
where $[K]$ and $[\partial\Delta^3]$ are the fundamental classes
determined by the chosen orientations.
\end{definition}

\begin{lemma}[Signed cell formula]
\label{lm:signed-cell-formula}
Let $K$ be a finite triangular cell decomposition of an oriented
$2$-sphere, and let
\[
    \Phi\colon |K|\longrightarrow\partial\Delta^3
\]
be a cellular map whose restriction to every triangular $2$-cell of
$K$ is a homeomorphism onto a face of $\partial\Delta^3$.

For every triangular $2$-cell $\sigma\in K^{(2)}$, let
\[
    \varepsilon_\Phi(\sigma)
    =
    \begin{cases}
        +1, & \Phi|_\sigma \text{ preserves orientation},\\
        -1, & \Phi|_\sigma \text{ reverses orientation}.
    \end{cases}
\]
Then, for every triangular face
$\tau\subset\partial\Delta^3$,
\begin{equation} \label{eq:top1}
    \deg\Phi
    =
    \sum_{\substack{\sigma\in K^{(2)}
                     \Phi(\sigma)=\tau}}
        \varepsilon_\Phi(\sigma).
\end{equation}
In particular, the signed count on the right-hand side is independent
of the choice of $\tau$.
\end{lemma}

\begin{proof}
Give every triangle of $K$ the orientation induced by the
orientation of $|K|$, and every face of $\partial\Delta^3$ the
boundary orientation. Let
\[
    z_K=\sum_{\sigma\in K^{(2)}}\sigma
\]
be the fundamental cycle of $K$. Then
\begin{equation} \label{eq:top2}
    \Phi_\#(z_K)
    =
    \sum_{\tau\in(\partial\Delta^3)^{(2)}}
    \left(
        \sum_{\substack{\sigma\in K^{(2)}
                         \Phi(\sigma)=\tau}}
            \varepsilon_\Phi(\sigma)
    \right)\tau.
\end{equation}

The chain $\Phi_\#(z_K)$ is a $2$-cycle on
$\partial\Delta^3$. Every integral $2$-cycle on the boundary of a
tetrahedron is an integer multiple of its fundamental cycle.
Consequently, all four coefficients in~\eqref{eq:top2} are equal to the same
integer $d$, and
\[
    \Phi_\#(z_K)
    =
    d\,z_{\partial\Delta^3}.
\]
By Definition~\ref{def:topological-degree}, this integer is
$d=\deg\Phi$. Comparing the coefficient of $\tau$ gives~\eqref{eq:top1}.
\end{proof}

\begin{corollary}
\label{cor:total-orientation-sign}
Under the assumptions of Lemma~\ref{lm:signed-cell-formula},
\begin{equation} \label{eq:top3}
    \sum_{\sigma\in K^{(2)}}
        \varepsilon_\Phi(\sigma)
    =
    4\deg\Phi.
\end{equation}
\end{corollary}

\begin{proof}
Each triangle of $K$ is mapped onto exactly one of the four faces
of $\partial\Delta^3$. By
Lemma~\ref{lm:signed-cell-formula}, the signed sum over the
preimages of each face is $\deg\Phi$. Summing over the four faces
gives~\eqref{eq:top3}.
\end{proof}

\subsection{Claw-tripod imbalance}

\begin{theorem}
Let $H$ be a connected cubic plane graph drawn with three pairwise
non-parallel edge directions $l_1,l_2,l_3$, with one edge in each
direction incident to every vertex. Then 
\[
N_{\rm claw}(H) - N_{\rm tripod}(H) = 0 \pmod 4.
\]
\end{theorem}

\begin{proof}
It is enough to prove the statement for connected $H$.

Color each edge of $H$ by $i\in\{1,2,3\}$ if it is parallel to
$l_i$. The union of any two color classes is a $2$-factor, so every
edge of $H$ lies on a cycle and $H$ has no bridges. Moreover, a cut
vertex in a connected cubic graph without bridges is impossible, since
one of the components obtained after its removal would be attached by a
single edge. Thus $H$ is $2$-connected, and its dual $H^*$ is a
triangular cell decomposition of $S^2$, with one triangular
$2$-cell $\sigma_v$ for every $v\in V(H)$.

Put
\[
    A=\mathbb Z_2^2=\{0,\gamma_1,\gamma_2,\gamma_3\},
    \qquad
    \gamma_1+\gamma_2+\gamma_3=0.
\]
Label a dual edge by $\gamma_i$ if the corresponding edge of $H$
has color $i$. The sum of the labels around every triangular
$2$-cell of $H^*$ is zero. Hence these labels form an
$A$-valued $1$-cocycle on $S^2$, which is a coboundary because
$H^1(S^2;A)=0$. We therefore obtain a face coloring
\[
    \kappa\colon F(H)\longrightarrow A
\]
such that
\[
    \kappa(F')-\kappa(F)=\gamma_i
\]
whenever $F$ and $F'$ are separated by an edge of color $i$.
Since every $\gamma_i$ is nonzero, this coloring is proper. Labeling
the vertices of $\Delta^3$ by the elements of $A$, the coloring extends over every triangular $2$-cell to a cellular map
\[
    \Phi_\kappa\colon |H^*|\longrightarrow\partial\Delta^3
\]
whose restriction to each triangular $2$-cell is a homeomorphism
onto a face of $\partial\Delta^3$.

Orient $H^*$ by the orientation of the plane. Relabel the directions
so that the six rays determined by $l_1,l_2,l_3$ occur
counterclockwise in the order $1,2,3,1,2,3$.
Among the eight possible choices of one ray from each line, the three
chosen rays lie in a common open half-plane precisely when their cyclic
color order is $1,2,3$. Thus the cyclic order is $1,2,3$ at a claw
and $1,3,2$ at a tripod.

If the cyclic order is $1,2,3$, then the colors of the three incident
faces, in counterclockwise order, are
\[
    \bigl(x,\;x+\gamma_2,\;x+\gamma_1\bigr)
\]
for some $x\in A$. If the cyclic order is $1,3,2$, they are
\[
    \bigl(x,\;x+\gamma_3,\;x+\gamma_1\bigr).
\]
Orient $\Delta^3$ by the ordered list $(0,\gamma_1,\gamma_2,\gamma_3)$
and give $\partial\Delta^3$ the induced boundary orientation. For
$x=0$, the first sequence, $(0,\gamma_2,\gamma_1)$,
has positive orientation, whereas the second, $(0,\gamma_3,\gamma_1)$,
has negative orientation. Translation by any element of $A$ is an
even permutation of the four vertices of the tetrahedron, so it
preserves orientation. Consequently,
\[
    \varepsilon_{\Phi_\kappa}(\sigma_v)
    =
    \begin{cases}
        +1, & v \text{ is a claw},\\
        -1, & v \text{ is a tripod}.
    \end{cases}
\]

Corollary~\ref{cor:total-orientation-sign} implies
\[
N_{\rm claw}(H)-N_{\rm tripod}(H) =   \sum_{v\in V(H)} \varepsilon_{\Phi_\kappa}(\sigma_v) = 4\deg\Phi_\kappa,
\]
which finishes the proof.
\end{proof}

\section{Questions and further directions}

Questions of the type addressed by Theorem~\ref{th:oblakov} naturally
arise when one studies how frequently the solution of a geometric
minimization problem is unique. In particular, excluding two
codirected solutions may help to control the exceptional terminal
configurations for which uniqueness fails.
In this case lengths of two different solutions coincide and 
continue to coincide when the terminals are perturbed.

Each question below has both a global and a local version. To avoid
repeating the two formulations, we use the term \emph{(local)
minimizer} to mean either a global minimizer or a locally minimizing
network, depending on the version under consideration.

\subsection{Higher-dimensional Steiner trees}

Proposition~\ref{melzakuniq} remains valid in $\mathbb R^d$, see~\cite{smith1992find}: for a
fixed terminal configuration and a fixed topology, there is at most
one locally minimal Steiner tree realizing that topology. It is
natural to ask whether the stronger uniqueness statement of
Theorem~\ref{th:oblakov} also extends beyond the plane.

\begin{question}
\label{quest:higher-dimensional-oblakov}
Let $P\subset\mathbb R^d$, $d\geq3$. Can there exist two distinct
(local) minimizers of the Steiner problem which are codirected at all
terminals?
\end{question}

By Proposition~\ref{melzakuniq}, such trees would necessarily have
different topologies. The first case to consider is $d=3$. The proof
of Theorem~\ref{th:oblakov} is essentially planar: it uses the plane
graph obtained from the symmetric difference, the cyclic order of
edges around a vertex, and the faces of the resulting embedding. None
of these ingredients has an immediate higher-dimensional analogue.

One may first restrict the question to full trees, for which every
terminal has degree one and every non-terminal vertex has degree
three.

\begin{question}
Can two distinct full (local) minimizers of the Steiner problem in
$\mathbb R^3$ be codirected at all terminals?
\end{question}

\subsection{The Gilbert--Steiner problem}

A second natural direction concerns the Gilbert--Steiner problem. We
recall its finite-network formulation. Let
\[
    P=\{x_1,\ldots,x_n\}\subset\mathbb R^d
\]
be a finite set of terminals, and prescribe numbers
\[
    b_1,\ldots,b_n\in\mathbb R\setminus\{0\},
    \qquad
    \sum_{i=1}^n b_i=0,
\]
where $b_i>0$ represents a source and $b_i<0$ represents a sink.

An admissible transport network is a finite directed straight-line
graph $G\subset\mathbb R^d$, whose vertex set contains $P$, together
with a positive flow $m_e$ on every directed edge $e$. The flows
satisfy the conservation law
\[
    \sum_{\substack{e\text{ leaving }v}}m_e
    -
    \sum_{\substack{e\text{ entering }v}}m_e
    =
    \begin{cases}
        b_i, & v=x_i,\\
        0,   & v\notin P.
    \end{cases}
\]
For a parameter $0<p<1$, the Gilbert--Steiner energy is
\[
    \mathcal E_p(G,m)
    =
    \sum_{e\in E(G)}
        m_e^p\,\mathcal H^1(e).
\]
The Gilbert--Steiner problem consists in minimizing
$\mathcal E_p$ among all admissible transport networks with the
prescribed boundary data $(P,b_1,\ldots,b_n)$. The concavity of the map
$m\mapsto m^p$ favors the merging of flows and produces branching
networks.

Recently it was shown that
Gilbert--Steiner minimizers have a combinatorial structure close to
that of Steiner trees: branching points are trivalent in the planar
case~\cite{cherkashin2025branching}, and also in arbitrary dimension in the range $p<1/2$~\cite{cherkashin2027branching}. These
are therefore the most natural settings in which to seek an analogue
of Theorem~\ref{th:oblakov}.

For two networks with the same boundary data, we say that they are
\emph{codirected at the terminals} if, for some $\varepsilon>0$,
their restrictions to $B_\varepsilon(P)$
coincide as directed weighted networks. Thus they have the same
oriented edges and the same transported flows in a neighborhood of
every terminal.

\begin{question}
\label{quest:gilbert-oblakov}
Can there exist two distinct (local) minimizers of the
Gilbert--Steiner problem, with the same boundary data, which are
codirected at all terminals?
\end{question}

The planar case is the natural first problem.

\begin{question}
For $0<p<1$, can two distinct planar (local) minimizers of the
Gilbert--Steiner problem be codirected at all terminals?
\end{question}

In higher dimensions, the same question appears particularly natural
for $p<1/2$, where all branching points are still trivalent.

Even in the trivalent case, the geometry differs essentially from the
ordinary Steiner problem. If $u_e$ denotes the unit vector along an
edge pointing away from an interior branching point, the first
variation gives the weighted equilibrium condition
\[
    \sum_{e\ni v} m_e^p u_e=0.
\]
Consequently, the angles at a branching point depend on the transported
flows and need not equal $2\pi/3$. Thus a minimal Gilbert--Steiner network does not generally have three fixed edge
directions, and the claw-tripod argument from the present paper cannot
be applied directly.

\paragraph{Acknowledgments.} I am grateful to Alexander Polyanskii, Yana Teplitskaya and Mikhail Shkolnikov for discussion.

\bibliography{main}
\bibliographystyle{plain}

\end{document}